\documentclass[12pt]{article}

\RequirePackage{amsthm,amsmath,amsfonts,amssymb,bbm,mathrsfs}
 \RequirePackage[numbers]{natbib}
\RequirePackage[colorlinks,citecolor=blue,urlcolor=blue]{hyperref}
\RequirePackage{graphicx}

\usepackage{algorithmicx}
\usepackage[ruled]{algorithm}
\usepackage{algpseudocode}

\theoremstyle{plain}

\newtheorem{theorem}{Theorem}[section]
\newtheorem{lemma}[theorem]{Lemma}
\newtheorem{proposition}[theorem]{Proposition}
\theoremstyle{definition}

\newtheorem{setting}{Setting}
\newtheorem{remark}{Remark}
\newtheorem{definition}{Definition}
\newtheorem{assumption}{Assumption}

\begin{document}

\title{On the Validity of Conformal Prediction for Network Data Under Non-Uniform Sampling}
\author{Robert Lunde \\
Department of Mathematics and Statistics \\ 
Washington University in St. Louis}

\maketitle


\begin{abstract}
We study the properties of conformal prediction for network data under various sampling mechanisms that commonly arise in practice but often result in a non-representative sample of nodes.  We interpret these sampling mechanisms as selection rules applied to a superpopulation and study the validity of conformal prediction conditional on an appropriate selection event. We show that the sampled subarray is exchangeable conditional on the selection event if the selection rule satisfies a permutation invariance property and a joint exchangeability condition holds for the superpopulation.  Our result implies the finite-sample validity of conformal prediction for certain selection events related to ego networks and snowball sampling.  We also show that when data are sampled via a random walk on a graph, a variant of weighted conformal prediction yields asymptotically valid prediction sets for an independently selected node from the population.   
\end{abstract}

\section{Introduction}
The statistical analysis of network data is now a common task in many disciplines, ranging from neuroscience to sociology.  A particularly important problem in this area is regression when the covariates include network information.  Arguably the most natural approach to modeling such data involves constructing network summaries for each node and including these terms as predictors in a regression model. While many options, ranging from linear regression to deep neural networks, are available for generating predictions, statistical inference is complicated by the non-standard dependence structure of the network summary statistics.

In other settings, conformal prediction, pioneered by Vladimir Vovk and colleagues in the 1990's and further developed by \citet{https://doi.org/10.1111/rssb.12021} and \citet{pmlr-v25-vovk12}, and \citet{barber2023conformal}, among others, is a method that allows valid uncertainty quantification under almost no assumptions on the fitted model. In particular, suppose that $(Y_1, X_1), \ldots, (Y_{n+1}, X_{n+1})$ are exchangeable.  Conformal prediction allows the construction of prediction set $\widehat{C}(X_{n+1})$ for the new test point $Y_{n+1}$ such that, for any pre-specified $\alpha \in (0,1),$
\begin{align*}
P(Y_{n+1} \in \widehat{C}(X_{n+1})) \geq 1-\alpha.
\end{align*}

In a recent manuscript, \citet{network-conformal} extended the validity of conformal prediction to settings in which the fitted model incorporates network information.  Suppose that $\hat{Z}_1, \ldots, \hat{Z}_{n+1}$ are nodal network summaries computed from both a $(n+1) \times (n+1)$ connection matrix $A$ and covariates $(X_1, \ldots, X_{n+1})$. Under a joint exchangeability condition on a regression array and a mild regularity condition on the network summary statistics (See Section \ref{sec:problem-setup-notation} for details), they show that $(Y_1,X_1,\hat{Z}_1), \ldots, (Y_{n+1},X_{n+1},\hat{Z}_{n+1})$ are exchangeable; thus, conformal prediction still offers the same guarantee for a new test point $Y_{n+1}$:
\begin{align*}
P(Y_{n+1}  \in \widehat{C}(X_{n+1}, \hat{Z}_{n+1})) \geq 1-\alpha.  
\end{align*}

While the above result suggests that conformal prediction is widely applicable in network-assisted regression problems, one implicit assumption is that the observed nodes are in some sense a representative sample.  However, in some settings, such an assumption is unlikely to hold.  For instance, in certain social science applications, one is interested in learning about the properties of individuals in rare, hard-to-reach populations; see for example, \citet{faugier-hard-reach}. To sample a sufficient number of such individuals, it is common to recruit participants that are known to some other members of the subpopulation. However, such referral strategies are likely to result in a non-representative sample of the subpopulation of interest.    Moreover, in online social media settings, it is often the case that third parties do not have access to the entire dataset and have to rely on crawling strategies to learn about the network; see for instance, \citet{6027868}. Random walks on graphs are not only biased towards higher degree nodes but also exhibit dependence between observations.       

In this paper, we consider whether conformal prediction remains valid for network-assisted regression under common sampling mechanisms for such data.  Surprisingly, we show that conformal prediction is still finite-sample valid for certain selection events related to ego and snowball sampling so long as the test point is chosen at random from the (non-representative) sample.  To prove these results, we view these sampling mechanisms as selection events acting on a larger population that is jointly exchangeable.  We then argue that if the random selection rule $\mathcal{S} \subseteq \{1, \ldots ,n\}$ satisfies a certain invariance property, then the subarray with indices given by $\mathcal{S}$ is \textit{conditionally} exchangeable given $\mathcal{S} = S$, leading to a guarantee of the form:
\begin{align}
\label{eq-conditional-coverage-guarantee}
P(Y_{test} \in \widehat{C}(X_{test}, \hat{Z}_{test}) \ | \ \mathcal{S} = S) \geq 1 - \alpha.
\end{align}
where the network summary statistics $(\hat{Z}_i)_{i \in S}$ are computed on the subarray with indices in $S$.  Thus, the mild joint exchangeability assumption on the regression array imposes enough structure to imply conditional exchangeability for certain selection events, leading to finite-sample valid prediction sets.  It should be noted that statistical inference under snowball sampling is a notoriously difficult problem; domain scientists are often rightfully hesitant to draw conclusions based on such data.  However, our results suggest that for certain selection events associated with snowball sampling, one can construct finite-sample valid prediction intervals for individuals in the sample under minimal assumptions.   

We also consider the question of whether it is possible to construct valid conformal prediction sets for observations not belonging to the selected set when nodes are sampled by a random walk on the graph.  We show that under certain conditions, a weighted conformal prediction procedure \citep{NEURIPS2019_8fb21ee7} that corrects for the bias of the random walk towards higher degree nodes achieves asymptotically valid coverage.  Moreover, we show that, if the underlying graphon model is sufficiently well-behaved, one only needs to traverse $m = \omega(\log n)$ nodes.  Our proof builds on arguments used in \citet{DBLP:conf/colt/ChernozhukovWZ18} and \citet{Chernozhukove2107794118} to establish the asymptotic validity of conformal prediction for time series.  However, some care is needed to uniformly bound the dependence properties of the random walk, which depend on the (random) topology of the underlying graph.  

The rest of the paper is organized as follows.  In Section \ref{sec:problem-setup-notation}, we introduce relevant notation and background.  In Section \ref{sec:conditonal-validity} we present finite sample validity results for selection rules satisfying a permutation invariance property.  In Section \ref{sec:asymptotic-validity-random-walk}, we state asymptotic validity results for a weighted conformal prediction procedure under random walk sampling.  A simulation study and a real-data example are considered in Section \ref{sec:experiments}.

\section{Problem Setup and Notation}
\label{sec:problem-setup-notation}
\subsection{Jointly Exchangeable Array}
We start by introducing notation.  We consider a superpopulation of size $n$.  Let $Y_1,\ldots, Y_n \in \mathbb{R}$ denote the corresponding response variables and $X_1, \ldots, X_n \in \mathbb{R}^{d}$ denote associated covariates.  Moreover, let $A$ be an $n \times n$ matrix, where $A_{ij}$ carries information about the connection between $i$ and $j$. In this paper, we consider the case where $A$ is binary.  While it is common to also assume that $A$ is undirected with no self-loops, in Section \ref{sec:snowball-sampling}, we consider directed graphs. When it is more convenient, we use the equivalent graph-theoretic notation, where $G$ is defined by its vertex set and edge set.  When the graph is undirected, the edge set consists of sets containing exactly two vertices; when the graph is directed, the edge set consists of ordered pairs of vertices.   

Now for $1\leq i, j \leq n$, let $V_{ij} = (Y_i,Y_j, X_i, X_j, A_{ij})$.   In addition, let $[n] = \{1, \ldots, n \}$, $\sigma:[n] \mapsto [n]$ denote a permutation function, and $\mathcal{L}(\cdot)$ denote the law or distribution of a random variable. Moreover, let $V^\sigma = (V_{\sigma(i) \sigma(j)})_{1 \leq i,j \leq n}$.  We make the following assumption, which was previously considered by \citet{network-conformal}:
\begin{assumption}
\label{assumption-exchangeability}
The array $(V_{ij})_{1 \leq i, j \leq n}$ is jointly exchangeable; that is, for any permutation function $\sigma$,
\begin{align}
\mathcal{L}(V^\sigma ) = \mathcal{L}(V).
\end{align}
\end{assumption}

Assumption \ref{assumption-exchangeability} is very general and includes many commonly used and natural network models.  For the reader's convenience, we present two classes of such regression frameworks below, which were previously discussed in the above reference:

\begin{setting}[Independent Triplets and Graphon Model]
\label{setting-independent-triples}
 Suppose that $(Y_1, X_1,\xi_1),$ $\ldots, (Y_{n}, X_{n}, \xi_{n})$ are i.i.d.\ triplets, where $\xi_1,\ldots, \xi_{n}$ are latent positions marginally uniformly distributed on $[0,1]$, and the adjacency matrix is generated as 
  \begin{align}
 \label{eq:sparse-graphon-model}
 A_{ij} = A_{ji} =  \mathbbm{1}(\eta_{ij} \leq \rho_n w(\xi_i, \xi_j) \wedge 1 ).  
 \end{align}
Here $\{\eta_{ij} \}_{1\leq i<j \leq n}$ is another set of i.i.d.\ $\mathrm{Uniform}[0,1]$ variables independent from all other random variables, $\rho_n$ controls sparsity of the network, and $w$ is a non-negative function symmetric in its arguments which satisfies $\int_0^1 \int_0^1 w(u,v) \ du \ dv =1$.  This model was first studied in the statistics literature by \citet{Bickel-Chen-on-modularity} and is inspired by representation theorems for jointly exchangeable arrays due to \citet{aldous-representation-array} and \citet{hoover-exchangeability}. As $n$ grows, it is natural to focus on the case $\rho_n \rightarrow 0$, since  most real world graphs are sparse, in the sense that they have $o(n^2)$ edges. 
\end{setting}

Note that the (sparse) graphon model subsumes many other commonly used classes of network models, including stochastic block models \citep{holland-sbm}, (generalized) random dot product graphs \citep{young-schneiderman-rdpg, 10.1111/rssb.12509}, and latent space models \citep{hoff-raftery-handcock-latent-space-model}.

Alternatively, one could consider the following setting where the response depends explicitly on neighborhood averages.  

\begin{setting}[Regression with Neighborhood and Node Effects]
\label{setting-neighborandnode}
Suppose that $(X_1, \xi_1,\epsilon_1) ,\ldots,$ $(X_{n}, \xi_{n},\epsilon_{n})$ are exchangeable and that $A$ is generated by the sparse graphon model  \eqref{eq:sparse-graphon-model}.   Let $\alpha_{ij}^{(k)}$ be a binary random variable equal to 1 if the shortest path from node $i$ to $j$ is of length $k$, and 0 otherwise.  Let 
\begin{align*}
\widetilde{D}_i^{(k)} = \sum_{j \neq i} \alpha_{ij}^{(k)}, \ \ 
\widetilde{X}_i^{(k)} = \frac{1}{  \widetilde{D}_i^{(k)}} \sum_{j \neq i} \alpha_{ij}^{(k)} X_j.
\end{align*}
Furthermore, let $\beta_{ij}$ is a weight function depending only on the length of the shortest path between nodes $i$ and $j$, and define the neighbor-weighted response:
\begin{align*}
\widetilde{Y}_i = \frac{1}{\sum_{j \neq i} \beta_{ij}} \sum_{j \neq i} \beta_{ij} Y_j.
\end{align*}
Now, suppose that $Y_i$ may be represented as: 
\begin{align*}
Y_i = f(X_i, \xi_i, \widetilde{Y}_i, \widetilde{D}_i^{(1)}, \ldots, \widetilde{D}_i^{(2n)},  \widetilde{X}_i^{(1)}, \ldots, \widetilde{X}_i^{(n)},\epsilon_i), 
\end{align*}
where $f$ is measurable function such that a unique solution exists almost surely.
This model is a nonparametric generalization of spatial autoregressive models for networks studied in, for example, \cite{10.2307/2298123}. When $\widetilde{Y}_i$ is included in the model, the value of $Y_i$ is determined endogenously.
\end{setting}

In this work, we often view a sampling mechanism as some selection rule applied to the indices of some population graph $V$.  Let $\mathcal{S}(V) \subseteq \{1,\ldots, n\}$ be a selection function.  For various choices of $\mathcal{S}$, we will be in interested in the properties of conformal prediction for the following array with a random index set:   
\begin{align*}
\mathcal{V} = (V_{ij})_{i,j \in \mathcal{S}}.
\end{align*}
If $\mathcal{S}$ is not in some sense drawn uniformly from $\{1, \ldots, n\}$, it will typically be the case that the random subarray $\mathcal{V}$ will not inherit exchangeability from $V$.  However, for certain selection events, the array is \textit{conditionally exchangeable} given $\mathcal{S}$, leading to valid conformal prediction for points within a potentially non-representative sample of nodes.  This idea is explored in Section \ref{sec:conditonal-validity}.

\subsection{Network Covariates}

In both Settings \ref{setting-independent-triples} and \ref{setting-neighborandnode}, since the latent positions $\xi_1, \ldots, \xi_{n}$ are unobservable, it is instead common to construct network summary statistics.   We fit a regression model to triplets of the form $(Y_i, X_i, \hat{Z}_i)$, 
where $\hat{Z}_i$ are local network statistics corresponding to node $i$; for concreteness, let $\hat{Z}_i \in \mathbb{R}^p$.  In this paper, when network covariates are included, we will assume that the network covariates take as input the subarray $\mathcal{U}_{\mathcal{S}(\mathcal{V})} = ( X_i, X_j, A_{ij})_{i,j \in \mathcal{S}(\mathcal{V})}$, where $\mathcal{S}(\mathcal{V})$ is random.  In what follows, we impose a symmetry condition for $\mathcal{U}_{S}$ for each possible choice of $S \subseteq \{1, \ldots, n\}$, where $S$ is treated as fixed.  To this end, consider the function $\zeta^{(k)}(\cdot)$, which takes as input a $k \times k$ subarray and constructs $k$ network statistics. We make the following assumption: 

\begin{assumption}
\label{assumption-permutation-invariance}
Suppose that for any $S \subseteq \{1,\ldots, n\}$, the corresponding network covariates $(\hat{Z}_1, \ldots, \hat{Z}_{k})$ satisfy:
\begin{align*}
(\hat{Z}_{\sigma(1)}, \ldots, \hat{Z}_{\sigma(k)}) = \zeta^{(k)}(\mathcal{U}_S^\sigma) \ \ a.s.
\end{align*}
for any permutation $\sigma:[1,\ldots, k] \mapsto [1,\ldots,k]$, where $|S| = k$.  
\end{assumption}

Assumption \ref{assumption-permutation-invariance} is very mild; essentially, we require that the values of the network summary statistics do not change when the regression array is re-labeled.  See \citet{network-conformal} for further discussion.  This condition is often satisfied for graph neural networks as well; see  \citet{Zargarbashi-gnn} and \citet{huang2023uncertainty}.  It should be noted that these works assume that the graph is fixed and the rest of the data are exchangeable.  One implication of our work is that it sheds more light on when such an assumption may be reasonable if one conditions on the graph after selection.   

\subsection{Split Conformal Prediction}
In this work, we consider split conformal prediction, which is a variant of conformal prediction that is often more computationally tractable than the originally proposed version of the method.  Let $S[i]$ denote the $i$th element of $S$ arranged in increasing order. Suppose that $m=|S| \geq 3$ and let $S \setminus S[m] = D_1 \cup D_2 $ be a partition of $S \setminus S[m]$ with at least one element in each set.  Without loss of generality, we treat $S[m]$ as the test point. Furthermore, let $\mathcal{D}_1  =(Y_i,X_i, \hat{Z}_i)_{i \in D_1}$ and $\mathcal{D}_2  =(Y_i,X_i, \hat{Z}_i)_{i \in D_2}$.  Define the non-conformity score $s(y,x,z; \mathcal{D}_1): (\mathbb{R} \times \mathbb{R}^d \times \mathbb{R}^p) \mapsto \mathbb{R}^+$, which measures how unusual a given point $(y,x,z)$ is. For split conformal prediction, the non-conformity score typically involves some function estimated on $\mathcal{D}_1$. For instance, a common choice is $|y - \hat{\mu}_n(x,z)|,$ where $\hat{\mu}_n(\cdot)$ is a regression function estimated on $\mathcal{D}_1.$  The procedure that we consider in Section \ref{sec:conditonal-validity} is given in Algorithm 1.     

\begin{algorithm}
\caption{Split Conformal Prediction}
\begin{flushleft} 
\hspace*{\algorithmicindent} \textbf{Input:} Data $(Y_i,X_i, \hat{Z}_i)_{i \in S\setminus S[m]}$, level $\alpha$, test point $(X_{test}, \hat{Z}_{test})$ \\ 
\hspace*{\algorithmicindent}
\textbf{Output:} Prediction Set $\widehat{C}(X_{test}, \hat{Z}_{test})$
\end{flushleft}
\begin{algorithmic}[1]
\State Split into two folds $\mathcal{D}_1 = (Y_i,X_i,\hat{Z}_i )_{i \in D_1}$, $\mathcal{D}_2 = (Y_i,X_i,\hat{Z}_i)_{i \in D_2}$.
\For{ $i \in D_2$ }
\State  $S_{i} \leftarrow s(Y_i,X_i, \hat{Z}_i \ ; \ \mathcal{D}_1)$
\EndFor
\State $d \leftarrow  (1-\alpha)(1+\frac{1}{|D_2|}) $ empirical quantile of $(S_i)_{i \in D_2}$
\State $\widehat{C}(X_{test}, \hat{Z}_{test}) \leftarrow \{y: s(y, X_{test}, \hat{Z}_{test}) \leq d \}$
\end{algorithmic}
\end{algorithm}

\section{Conditional Exchangeability via Invariant Selectors}
\label{sec:conditonal-validity}
In this section, we adopt a selective inference perspective (see for example, \citet{fithian2017optimal}) to establish the validity of conformal prediction for certain selection rules.  For $S \subseteq \{1, \ldots, n\}$, consider the subgroup of permutations:
\begin{align}
   \Sigma_S = \left\{ \sigma: [n] \mapsto [n] \ \bigr\rvert \ \ \sigma(i) \in S \ \ \forall \  i \in S, \  \sigma(i) = i \ \ \forall \  i \not \in S  \right\}.
\end{align}
We introduce the following notion:

\begin{definition}[Invariant Selector]
The selection rule $\mathcal{S}(V)$ is an invariant selector if for each $S \subseteq \{1, \ldots, n\}$ such that $P(\mathcal{S} = S) >0$,
\begin{align*}
\{ \mathcal{S}(V^{\sigma}) = S\} = \{ \mathcal{S}(V) = S\} \ \  \forall \ \sigma \in \Sigma_S. 
\end{align*}
\end{definition}
The notion of an invariant selector is different from permutation invariance condition stated in Assumption \ref{assumption-permutation-invariance}.  We require that the indices of the selected nodes remain fixed when a subgroup of permutations is applied to $V$ and that the selection event itself does not change. Some of the most natural selection events are related to the presence or absence of edges.  In these cases, one can treat the selector as purely a function of $A$.  Ego sampling, considered in Section \ref{sec:ego}, is an illustrative example, but the idea is more general.

\begin{remark}
One may consider a slightly weaker notion of invariance.  The proof of Theorem \ref{theorem-conditional-exchangeability} still goes through under the assumption that the symmetric set difference, which we denote $\ominus$, has measure zero; that is, we can require that an invariant selector $\mathcal{S}(V)$ satisfies:
\begin{align*}
P(\{ \mathcal{S}(V^{\sigma}) = S\} \ominus \{ \mathcal{S}(V) = S\}) = 0 \ \ \forall \ \sigma \in \Sigma_S
\end{align*}
for all $S$ such that $P(\mathcal{S}(V) = S) >0.$ 

\end{remark}
The next result establishes that jointly exchangeable arrays are also conditionally exchangeable if one conditions on an invariant selector.  

\begin{theorem}[Conditional Exchangeability via Invariant Selectors]
\label{theorem-conditional-exchangeability}
Suppose that $(V_{ij})_{i,j \in \{1, \ldots, n\}}$ satisfies Assumption \ref{assumption-exchangeability} and that $\mathcal{S}(V)$ is an invariant selector. Then, $\mathcal{V}$ is jointly exchangeable conditional on $\mathcal{S}$; that is, for each $S$ such that $P(\mathcal{S}(V) = S) >0$:
\begin{align*}
\mathcal{L}( \mathcal{V}^{\sigma} \ | \ \mathcal{S} = S)  = \mathcal{L}( \mathcal{V} \ | \ \mathcal{S} = S) \ \  \forall \ \sigma \in \Sigma_S.
\end{align*}
\end{theorem}
While the proof of the above result is elementary, the theorem has substantial implications for conformal prediction with network data.  If the test point is randomly chosen after selection by an invariant selector, then conformal prediction is still finite-sample valid. We formally state this result below.  As previously discussed, for ease of exposition, we assume that the test index is the largest one belonging to $S$.  To the best of our knowledge, this is the first time an invariant set has been exploited in the context of conformal prediction.  
\begin{theorem}[Validity of Conformal Prediction after Selection By an Invariant Selector]
\label{theorem-conditional-conformal-theorem}
Suppose that Assumptions \ref{assumption-exchangeability} and \ref{assumption-permutation-invariance} hold and that the selection rule $\mathcal{S}(V)$ is an invariant selector.  Then, $s(Y_i, X_i, \hat{Z}_i ; \mathcal{D}_1 )_{i \in \mathcal{D}_2}$ is conditionally exchangeable given $\mathcal{S}$. 
Moreover, for each $S$ such that $P(\mathcal{S} = S) > 0:$        
\begin{align*}
P(Y_{test} \in \widehat{C}(X_{test}, \hat{Z}_{test}) \ | \ \mathcal{S} = S) \geq 1 - \alpha.
\end{align*}
\end{theorem}
One surprising aspect of our result above is that the network summary statistics may be computed on the subarray with indices in $S$.  These statistics may not even asymptotically behave like their counterparts computed on the entire graph; all that is required is the symmetry condition stated in Assumption \ref{assumption-permutation-invariance}.        In the following subsections, we will show that common sampling mechanisms may be associated with an invariant selector, thus yielding finite-sample validity of conformal prediction in these cases.  Our conditional exchangeability results may also be of independent interest.

\subsection{Ego Networks}
\label{sec:ego}
An ego network corresponds to a subgraph induced by vertices adjacent to a fixed ego node.  In certain social science applications, ego networks can be of fundamental interest since one may be interested in the personal network of a given person rather than the structure of the entire network; see \citet{perry_pescosolido_borgatti_2018} for an overview of egocentric network analysis.  While these networks are of substantial interest in their own right, ego sampling may also have certain advantages over uniform vertex sampling for prediction tasks. It is possible that the non-conformity scores are less variable within an ego network due to homophily, which may lead to tighter prediction sets.  Moreover, it is likely that ego networks are more densely connected than the entire network, leading to more informative network summary statistics.   
        
We now discuss a selection rule associated with an ego network rooted at node $i$. For any set $S$ that does not include the root $i$, the selection rule associated with the ego network for node $i$ is given by the equivalence:
\begin{align}
\mathcal{S}_{\mathcal{N}(i)}(A) = S \iff  A_{ij} = 1 \ \  \forall \ j \in S, \ A_{ij} = 0 \ \  \forall \ j \not\in S.
\end{align}
For $S$ such that $i \in S$, we set $P(\mathcal{V}_{\mathcal{N}(i)} = S) = 0$.
We have the following proposition:
\begin{proposition}
The selection rule $\mathcal{S}_{\mathcal{N}(i)}$ is an invariant selector.   
\end{proposition}
To see that this is case, observe that, for $S$ such that $i \not\in S$, 
\begin{align*}
\{ \mathcal{S}_{\mathcal{N}(i)}(A^\sigma) = S    \} &= \{ A_{ij}^\sigma = 1 \ \  \forall \ j \in S, \ A_{ij}^\sigma = 0 \ \ \forall \ j \not\in S \}
\\ &=   \{ \mathcal{S}_{\mathcal{N}(i)}(A) = S \}.
\end{align*}

Note that the permutation invariance property would not hold for all $\sigma \in \Sigma_S$ if the root node $i$ is included in the ego network.  Intuitively, the network $\mathcal{V}$ associated with $\mathcal{S}_{\mathcal{N}(i)}$ is conditionally exchangeable since labels within the ego network are still uninformative even if the ego network is not representative of the population.

\subsection{Snowball Sampling}
\label{sec:snowball-sampling}
We now consider snowball sampling, which may be considered a generalization of ego sampling.  For a historical account of snowball sampling, see \citet{handcock-snowball}. In snowball sampling, one starts with a set of $m_0$ seed nodes, who then each recruit new participants from their neighbors.  The number of nominated participants is typically either some fixed number $r$ or all neighboring nodes.  In the next wave, the newly recruited subjects nominate new participants and this process is iterated until some stopping criteria is met.  In our setting, the stopping criteria will be immaterial since we consider a hypothetical referral matrix that contains referrals if a subject were to participate in the study, but we only condition on events that we have observed. Since referrals are potentially asymmetric, it will be more natural to consider directed graphs, which still fall under the framework of joint exchangebility; see for example, \citet{diaconis-janson-exchangeable-graph}.  

It should be noted that implicit in the assumption of such a referral matrix is that referrals are not directly influenced by people who recruited the subject.  However, our setup is still more general than many others in literature.  Respondent driven sampling (RDS), introduced by \citet{10.2307/3096941}, is closely related to snowball sampling and attempts to construct asymptotically unbiased estimators for population quantities by modeling inclusion probabilities.  However, as noted in \citet{10.1214/18-AOS1700}, it is often assumed that referrals are chosen uniformly from the set of neighbors, which we will not need to assume. In our setup, one may include both the referral matrix and adjacency matrix in the array $V$ and not specify the relationship between them.  Moreover, while our inference goal is different, our theory implies finite-sample validity for conformal prediction sets under mild conditions opposed to the asymptotic guarantees for these estimators.         

In the case of fixed referrals, we require that, for some $r \geq 1$:
\begin{align}
\sum_{j=1}^n A_{ij}= r \ \  \forall \ i \in \{1, \ldots, n \}.
\end{align}
While the Aldous-Hoover theorem suggests that such restrictions are not possible for jointly exchangeable infinite arrays, there are jointly exchangeable finite arrays that satisfy this property.  As an  example, we construct one very natural data generating process below that is closely tied to the graphon setup. Suppose that  $\xi_1, \ldots, \xi_n \sim \mathrm{Uniform}[0,1]$, $w: [0,1]^2 \mapsto \mathbb{R}$,  and let $W_{ij} = g( \eta_{ij}, w(\xi_i, \xi_j))$ be continuously distributed edge weights, where $\eta_{ij} \sim \mathrm{Uniform}[0,1]$ are generated independently of other random variables.  

Suppose that $A_{ii} =0$ and for $i \neq j$,
\begin{align}
\label{eq:fixed_degree_exchangeable}
A_{ij} = \begin{cases}
1 & \text{rank}(W_{ij}, \{W_{i1}, \ldots, W_{in}\} ) \leq r 
\\ 0 & \text{otherwise},
\end{cases}
\end{align}
where  $\mathrm{rank}(\cdot, \cdot)$ is the position of an element in a list arranged in descending order and $W_{ii}$ is excluded from the computation of the rank.  For completeness, we prove the following proposition in the Appendix:
\begin{proposition}
\label{proposition-referral-exchangeable}
The matrix $A$ defined in (\ref{eq:fixed_degree_exchangeable}) is jointly exchangeable.     
\end{proposition}

For snowball sampling, one of the most natural classes of sets to consider are waves.  In words, the $k$th wave consists of all people who were nominated for the first time by people in the the $(k-1)$th wave. We now define this set mathematically. Let $M_0$ denote the set index of the seed nodes.  We define waves iteratively.  Let $M_0 \subseteq \{1,\ldots, n\}$ denote the indices of the seed nodes. The zeroth wave is $M_0$ and the first wave is given by the set:
\begin{align*}
\mathcal{W}(M_0,1) &= \left\{ j |  A_{ij} = 1 \text{ for some } i \in M_0  \right\} \setminus M_0
\\ &= \bigcup_{i \in M_0} \{ j \ | \  A_{ij} = 1\} \setminus M_0.
\end{align*}
The $k$th wave is defined iteratively as:
\begin{align*}
\mathcal{W}(M_0,k) &= \bigcup_{i \in \mathcal{W}(M_0, k-1)}\{ j \ | \  A_{ij} = 1\} \biggr\backslash \left(\bigcup_{0 \leq t \leq k-1} \mathcal{W}(M_0,t) \right).
\end{align*}
The selection rule associated with the $k$th wave of snowball sampling is given by the equivalence:
\begin{align}
\mathcal{S}_{ \mathcal{W}(M_0,k)} = S \iff  \mathcal{W}(M_0,k) = S.
\end{align}
The following statement is proven in the Appendix:
\begin{proposition}
\label{proposition-snowball-invariant}
The selection rule $\mathcal{S}_{ \mathcal{W}(M_0,k)}$ is an invariant selector. 
\end{proposition}

It is also possible to consider unions of $k$-hop neighborhoods starting from the set of seed nodes $M_0$.   For $k \leq 1$, let $\mathcal{K}(M_0, k) = \mathcal{W}(M_0, k)$ and and for $k > 1 $, define:
\begin{align*}
\mathcal{K}(M_0, k) = \bigcup_{l=1}^{k-1} \bigcup_{i \in \mathcal{K}(M_0, l )} \{ j \ | \ A_{ij} = 1 \} \setminus M_0.
\end{align*}
Analogous to before, define the selection rule $\mathcal{S}_{\mathcal{K}(M_0,k)}$ by the following equivalence:
\begin{align*}
\mathcal{S}_{\mathcal{K}(M_0,k)} = S \iff \mathcal{K}(M_0, k) = S. 
\end{align*}
The following statement is proved in the Appendix:
\begin{proposition}
\label{proposition-k-hop-invariant}
$\mathcal{S}_{\mathcal{K}(M_0,k)}$ is an invariant selector. 
\end{proposition}

While it may initially seem that using $\mathcal{S}_{\mathcal{K}(M_0,k)}$ would result in tighter prediction intervals since the resulting sample size is no smaller than the sample size corresponding to $\mathcal{S}_{\mathcal{W}(M_0,k)}$, it should be noted that the nodes selected by $\mathcal{S}_{\mathcal{W}(M_0,k)}$  may be more similar to each other. See Section \ref{sec-snowball-sampling} for some experimental evidence of this phenomenon.




\section{Asymptotic Validity Under Random Walk Sampling}
\label{sec:asymptotic-validity-random-walk}

Another approach to sampling nodes involves traversing along edges via a random walk.  In some sense, this is a depth-first sampling strategy that is quite different in spirit from snowball sampling, which typically involves more breadth. Instead of aiming for a form of validity conditional on a selection event, we consider the more ambitious problem of using information from nodes visited by the random walk to generate (asymptotically) valid prediction sets for a new, randomly selected test point.  We show that, while the random walk is biased towards nodes with higher degrees, one may correct for this bias using appropriate weights, leading to asymptotically valid conformal prediction sets.  Moreover, we show that, for a large class of models in Setting \ref{setting-independent-triples}, one may start the random walk at an arbitrary point and one may only need to traverse $\omega(\log n)$ nodes to achieve asymptotically valid prediction sets.

We now elaborate on our conformal prediction algorithm over the random walk.   Let $\mathcal{X}_k \in \{1, \ldots, n\}$ denote the label of the node visited on step $k$ of the walk.  Recall that a random walk on a graph is a Markov chain in which the transition kernel is given by:
\begin{align*}
\tilde{P}_n(\mathcal{X}_{k+1} = j\ | \ \mathcal{X}_k =  i )  = \begin{cases}
\frac{1}{D_i} & j \in \mathcal{N}(i) \\ 
0 &  j \not \in \mathcal{N}(i),
\end{cases}
\end{align*}
where $D_i = \sum_{j\neq i} A_{ij}$ is the degree of node $i$ and $\mathcal{N}(i) = \{ j \in \{1, \ldots, n \} \ | \  A_{ij} = 1 \}$ is neighbor set of node $i$.  Note that the kernel depends on $V$ and is itself random; we also use the subscript $n$ to emphasize that the probability measure also changes due to the triangular array setup.  As a technicality, we assume that the random variables determining the choice of neighbor at each step are independent of all other random variables.  In practice, when using a random walk, it is uncommon to store information that would be needed to construct an induced subgraph.  Thus, we will not consider network covariates and will only conventional covariates as inputs into a model. 

 We consider a split conformal prediction procedure, where the first $m+1$ observations $\tilde{D}_1 = \{0, \ldots, m\}$ are used to fit a model and the last $m$ points $\tilde{D}_2 = \{m+1, \ldots, 2m \}$ are used construct non-conformity scores.  Let $(\tilde{Y}_0, \tilde{X}_0), \ldots (\tilde{Y}_{2m}, \tilde{X}_{2m})$ denote the values of the $(Y,X)$ pairs visited by the random walk, where $\widetilde{\mathcal{D}}_1 = (\tilde{Y}_i, \tilde{X}_i)_{i \in \tilde{D}_1}$.  Moreover, let $\tilde{S}_i = s(\tilde{Y}_{m+i}, \tilde{X}_{m+i}; \widetilde{\mathcal{D}}_1)$ denote non-conformity scores computed on $\tilde{D}_2.$
 We consider a prediction set of the form:
\begin{align}
\label{eq-conformal-prediction-set}
\widetilde{C}_m(X') = \left\{ y \in \mathbb{R} \ \bigr\rvert \ \frac{1}{m}\sum_{i=1}^{m} \nu(\mathcal{X}_{m+i}) \mathbbm{1}(\tilde{S}_i \leq s(y,X'; \widetilde{\mathcal{D}}_1)) \leq 1-\alpha   \right\},
\end{align}
where the weight function $\nu:\{1,\ldots, n\} \mapsto \mathbb{R}$ is given by:
\begin{align}
\nu(j) = \frac{1}{n \pi(j)} = \frac{2|E|}{nD_{j}}.
\end{align}
Above, $\pi(j)$ is the probability of visiting node $j$ according to the stationary measure associated with the random walk; note that this measure is also random in our setup.  For a simple random walk, it is well-known that $\pi(j) = D_j/2|E|$; see for example, \citet{lovasz1993random}.  When traversing a graph using a random walk, it is common to assume that we observe all neighbors of the current node; thus, requiring knowledge of a degree of a given node is not restrictive.  However, to construct these weights, we do require some global information about the graph: namely, we need to know the total number of nodes and edges.  For most datasets found in network data repositories, this information is usually known, but this global information may not be known precisely in web-crawling applications.

It should be noted that the weights in our weighted conformal prediction procedure are not normalized, which is more typical in the literature.  We view the quantity $\frac{1}{m} \sum_{i=1}^m \nu(\mathcal{X}_{m+i}) \mathbbm{1}(\tilde{S}_i \leq x)$ as a weighted empirical process (c.f. \citet{b3105f83-fbb8-3bc3-a28f-c0294831bf87}), where the weights are used to control the expectation conditional on $V.$  Furthermore, since the coverage guarantee will be asymptotic, in (\ref{eq-conformal-prediction-set}), we choose a prediction set that is more natural in the limit. We are now ready to state a theorem that establishes conditions under which our weighted conformal prediction procedure yields asymptotically valid prediction sets.

\begin{theorem}
\label{thm-conformal-prediction-random-walk}
Suppose that we are in Setting \ref{setting-independent-triples}, where the sparse graphon model satisfies $\rho_n = \omega(\log^{3/2} n / n)$ and $0 < c\leq w(x,y) \leq C < \infty$ almost surely.    
Furthermore, let   $(\tilde{Y}_0, \tilde{X}_0), \ldots, (\tilde{Y}_{2m}, \tilde{X}_{2m})$ be observations associated with the random walk, where $m = \omega(\log n)$ and the initial node is chosen by an arbitrary measurable function $f(V, \eta)$, where $\eta$ is independent of all other random variables.  Further suppose that the following conditions hold:
\begin{itemize}
\item[(A1)] \label{assumption-score-consistency} (Score Consistency) There exists a score function $s_*(x,y)$ that does not depend on $\widetilde{\mathcal{D}}_1$ such that, for any $\epsilon > 0$, 
\begin{align*}
\frac{1}{m} \sum_{i=1}^{m} P\left( |\tilde{S}_i - s_*(\tilde{Y}_i, \tilde{X}_i)| > \epsilon \right) = o(1), \ \ P\left(|s(Y',X'; \widetilde{\mathcal{D}}_1) -s_*(Y',X')| > \epsilon \right) = o(1).
\end{align*}
\item[(A2)]\label{assumption-score-density} (Score Density) $S_i = s_*(Y_i, X_i)$ is continuously distributed with density upper bounded by $K < \infty.$
\end{itemize}
Then, for the weighted split conformal prediction procedure defined in (\ref{eq-conformal-prediction-set}) and  a randomly chosen observation $(Y',X')$:
\begin{align*}
P\left(Y' \in \widetilde{C}_m(X') \right) = 1-\alpha + o(1).
\end{align*}
\end{theorem} 

In the above theorem, we  require that the non-conformity score can be approximated by a non-conformity score that does not depend on $\widetilde{\mathcal{D}}_1$; a similar condition is imposed by \citet{DBLP:conf/colt/ChernozhukovWZ18}.  This condition typically holds if the estimated function is converging in a suitable sense to a limiting quantity, even if the model is misspecified. While this property needs to hold with respect to the random walk, under the conditions in our theorem, the Markov chain is rapidly mixing so dependence itself will often be weak enough to not be an issue.  While not necessary, one may also choose to discard some training points during a burn-in phase for the Markov chain.  It should be also noted that this condition can be replaced by a stability condition at the expense of longer proofs.  Finally, one may wonder whether a function learned on the random walk would lead to a non-conformity score with good properties for an independent test point.  It should be noted that functions commonly used to construct non-conformity scores are functionals of the distribution of $Y \ | \ X$ and are thus not strongly influenced by a bias towards higher degree nodes.    

In addition, our theorem above suggests that, even though the underlying graph is random, one does not need to carefully choose the starting point or the number of nodes traversed for each instantiation of the graph. As an intermediate step, we study the rate of convergence of a random walk on a random graph to its stationary distribution.  In particular, for a large class of sparse graphon models, we show that the random walk mixes quite fast with high probability from any initial point. We state this result below.  While random walks on Erd\H{o}s-R\'{e}nyi graphs are well-studied even in very sparse regimes (see for example, \citet{10.1214/17-AOP1189} and references therein), to the best of our knowledge, a result of this type was not previously known under the sparse graphon model.      
\begin{theorem}[Mixing Rate of Random Walks on Sparse Graphons]
\label{thm:mixing-time}
Suppose that $\{A^{(n)}\}_{n \geq 1}$ is a sequence of graphs generated by the sparse graphon model (\ref{eq:sparse-graphon-model}), where the graphon satisfies $0 < c \leq w(x,y) \leq C < \infty$ almost surely and $\rho_n = \omega(\log^{3/2} n / n).$ Then, there exists constants $0 < K < \infty$ and $0 \leq \gamma < 1$ depending only on $w(\cdot,\cdot)$ such that:
\begin{align*}
P\left(\forall{t} \geq 1, \ \max_{x_0 \in \{1 ,\ldots, n\} } \|\widetilde{P}_n^t(\cdot \ | \ \mathcal{X}_0 = x_0) -  \pi(\cdot) \|_{TV} \leq K \sqrt{n} \gamma^t\right)  \rightarrow 1.
\end{align*}
\end{theorem}
We now discuss some of the key steps to proving the above theorem. It is well-known that the mixing rate of a random walk on a graph is closely related to the eigenvalues of the normalized adjacency matrix $\mathscr{A} = D^{-1/2} A D^{-1/2},$ where $D$ is a diagonal matrix for which $D_{ii}$ is equal to $D_i$.  In particular, let  $\lambda_{-1}(\mathscr{A})$ denote the negative eigenvalue of $\mathscr{A}$ with the largest magnitude and define $\gamma(\mathscr{A}) = \max\{\lambda_2(\mathscr{A}), |\lambda_{-1}(\mathscr{A})|\}$. The following bound for the rate of convergence of a random walk is well-known (c.f. \citet{lovasz1993random} Theorem 5.1):
\begin{align*}
\|\widetilde{P}_n^t(\ \cdot \ | \ \mathcal{X}_0 = i) - \pi(\cdot) \|_{TV} \leq \frac{ \gamma(\mathscr{A})^t}{\sqrt{\pi(i)}}.
\end{align*}
 To establish a uniform bound on the eigengap $1-\gamma(\mathscr{A})$ on a high probability set, we study the concentration of eigenvalues of $\mathscr{A}$ around the corresponding eigenvalues of the following integral operator $T_w: L^2[0,1] \mapsto L^2[0,1]:$
\begin{align*}
T_wf = \int_0^1 \frac{w(x,y)}{\sqrt{\mathfrak{D}(x) \ \mathfrak{D}(y)}} f(y) \ dx,
\end{align*}
where $\mathfrak{D}(x) = \int_0^1 w(x,y) \ dy$ is related to the degree distribution.  If the eigenvalues of $\mathscr{A}$ are concentrating in an appropriate sense, then an eigengap for the integral operator would imply a high probability bound on $\gamma(\mathscr{A}).$ The concentration of the normalized adjacency matrix around a typical matrix is well-studied (see for example, the review article by \citet{le-concentration-random-graph}), but concentration results with respect to the eigenvalues of the integral operator are less well-known. The most closely related work that we are aware of is that of \citet{oliveira2010concentration}, who states a concentration result for the normalized Laplacian in the discussion section of the paper.  Since a form of concentration for the eigenvalues of $\mathscr{A}$ around the corresponding eigenvalues of $T_w$ is critical to our argument, for completeness, in the Appendix we provide a proof of this property.  Moreover, by appealing to a generalization of the Perron-Frobenius theorem to Banach spaces known as the Krein-Rutman theorem, stated in Proposition \ref{proposition-krein-rutman} of the Appendix, we argue that the condition $0<c \leq w(x,y) \leq C < \infty \ a.s.$ is enough to ensure the existence of an eigengap for the integral operator. In the dense case, we have recently found that \citet{doi:10.1137/20M1339246} make an argument for a gap between $\lambda_1$ and $\lambda_2$ of the normalized Laplacian operator, but in comparison, our approach also implies a gap with respect to the negative eigenvalue with largest magnitude.       

\section{Experiments}
\label{sec:experiments}

\subsection{Snowball Sampling}
\label{sec-snowball-sampling}
To thoroughly put our theory to the test, we consider a simulation study with a fairly complicated data generating process that is still jointly exchangeable.   In particular, we study a spatial autoregressive model that also depends on latent positions of a rank 3 RDPG model. Let $n=2000$ be the superpopulation size and  $\nu_n = 5n^{-0.25}$ be the sparsity level.  For the network component of the model, similar to \citet{network-conformal} we consider a truncated eigendecomposition of the graphon $w(x,y) = \min(x,y)$, which has the eigenvalue-eigenfunction pairs:
\begin{align*}
\lambda_k = \left(\frac{2}{(2k-1) \pi}\right)^2, \ \ \  \phi_k(x) = \sin\left(\frac{(2k-1)\pi x}{2} \right).   
\end{align*}
See \citet{xu2018rates} for a derivation of these spectral properties. We generate $\xi_1, \ldots, \xi_n \sim \mathrm{Uniform}[0,1]$ and let $Z_{ik} = \sqrt{\lambda_k \phi_k(\xi_i)}$ for $k=1,2,3$, so that:
\begin{align*}
P_{ij} = P(A_{ij} = 1 \ | \ \xi_i, \xi_j) = \nu_n Z_i^T Z_j.
\end{align*}

For the regression component, consider the following model:
\begin{align*}
Y_i = 2+ 0.5 \tilde{Y}_i + 2 \tilde{X}_i + 10 Z_{i1} + 20 Z_{i2} + 10 Z_{i3} + 4 X_i + \epsilon_i,  
\end{align*}
where $\tilde{Y}_i$ and $\tilde{X}_i$ are the neighbor-weighted response and covariate, respectively, $X_i = U_i + 4Z_{i1}$ for $U_i \sim \mathrm{Uniform}[-2,1]$, and $\epsilon_i \sim N(0,1)$.

We consider two waves of snowball sampling with three different referral schemes.  The first referral scheme involves 3 randomly chosen seed nodes and at each step, all neighbors of nodes in the previous wave are included.  In the second referral scheme, we consider 10 seed nodes and 10 referrals for each node, where referrals are given by the 10 largest elements in each row of $P$ similar to the data generating process in (\ref{eq:fixed_degree_exchangeable}).  For the third referral scheme, we consider 20 random seeds with 5 referrals for each participating node.  For each of these schemes, we also study the union of 2-hop neighborhoods.  

We consider two fitted models.  The first model $\hat{\mu}_1$ is a smoothing spline fit only using the covariate $X$.  The second model $\hat{\mu}_2$ is a linear model with covariates including $X$ and estimates for the random dot product graph embedding (see for example, \citet{doi:10.1080/01621459.2012.699795} or Example 2 of \citet{network-conformal}).  For schemes 2 and 3, we construct the embeddings using the (induced) adjacency matrix rather than the referral matrix.   Note that both models are misspecified, but our theory still predicts finite sample validity for both. 
Moreover, it is not clear what population quantity the RDPG embeddings are estimating in snowball samples, but these statistics do satisfy Assumption \ref{assumption-permutation-invariance} and therefore, we expect conformal prediction to retain its finite sample validity.  For simplicity, for both models we use a standard regression conformity score of the form $|y- \hat{\mu}(x,z)|$, with $\alpha =0.1$.

\begin{table}[h]
 \label{table-snowball-sims}
 \centering
\begin{tabular}{ccccccc}
\multicolumn{1}{l}{}           & \multicolumn{1}{l}{}        & \multicolumn{2}{c}{Smoothing Spline}                        &                       & \multicolumn{2}{c}{Linear Regression}                      \\ \cline{3-4} \cline{6-7} 
                               & \multicolumn{1}{c|}{}       & \multicolumn{1}{c|}{Coverage} & \multicolumn{1}{c|}{Width}  & \multicolumn{1}{c|}{} & \multicolumn{1}{c|}{Coverage} & \multicolumn{1}{c|}{Width} \\ \cline{1-4} \cline{6-7} 
\multicolumn{1}{|c|}{Scheme 1} & \multicolumn{1}{c|}{Wave 1} & \multicolumn{1}{c|}{0.896}    & \multicolumn{1}{c|}{5.819}  & \multicolumn{1}{c|}{} & \multicolumn{1}{c|}{0.890}    & \multicolumn{1}{c|}{5.197} \\
\multicolumn{1}{|c|}{}         & \multicolumn{1}{c|}{Wave 2} & \multicolumn{1}{c|}{0.876}    & \multicolumn{1}{c|}{6.817}  & \multicolumn{1}{c|}{} & \multicolumn{1}{c|}{0.908}    & \multicolumn{1}{c|}{6.237} \\
\multicolumn{1}{|c|}{}         & \multicolumn{1}{c|}{Hop 2}  & \multicolumn{1}{c|}{0.906}    & \multicolumn{1}{c|}{6.544}  & \multicolumn{1}{c|}{} & \multicolumn{1}{c|}{0.906}    & \multicolumn{1}{c|}{5.846} \\ \cline{1-4} \cline{6-7} 
\multicolumn{1}{|c|}{Scheme 2} & \multicolumn{1}{c|}{Wave 1} & \multicolumn{1}{c|}{0.916}    & \multicolumn{1}{c|}{10.617} & \multicolumn{1}{c|}{} & \multicolumn{1}{c|}{0.906}    & \multicolumn{1}{c|}{6.006} \\
\multicolumn{1}{|c|}{}         & \multicolumn{1}{c|}{Wave 2} & \multicolumn{1}{c|}{0.917}    & \multicolumn{1}{c|}{6.607}  & \multicolumn{1}{c|}{} & \multicolumn{1}{c|}{0.931}    & \multicolumn{1}{c|}{4.477} \\
\multicolumn{1}{|c|}{}         & \multicolumn{1}{c|}{Hop 2}  & \multicolumn{1}{c|}{0.917}    & \multicolumn{1}{c|}{7.552}  & \multicolumn{1}{c|}{} & \multicolumn{1}{c|}{0.927}    & \multicolumn{1}{c|}{4.597} \\ \cline{1-4} \cline{6-7} 
\multicolumn{1}{|c|}{Scheme 3} & \multicolumn{1}{c|}{Wave 1} & \multicolumn{1}{c|}{0.916}    & \multicolumn{1}{c|}{6.982}  & \multicolumn{1}{c|}{} & \multicolumn{1}{c|}{0.928}    & \multicolumn{1}{c|}{5.540} \\
\multicolumn{1}{|c|}{}         & \multicolumn{1}{c|}{Wave 2} & \multicolumn{1}{c|}{0.922}    & \multicolumn{1}{c|}{4.545}  & \multicolumn{1}{c|}{} & \multicolumn{1}{c|}{0.902}    & \multicolumn{1}{c|}{4.286} \\
\multicolumn{1}{|c|}{}         & \multicolumn{1}{c|}{Hop 2}  & \multicolumn{1}{c|}{0.926}    & \multicolumn{1}{c|}{5.369}  & \multicolumn{1}{c|}{} & \multicolumn{1}{c|}{0.904}    & \multicolumn{1}{c|}{4.530} \\ \cline{1-4} \cline{6-7} 
\end{tabular}
\caption{Average coverage and width for snowball sampling under the spatial autoregressive model computed over 500 runs for $\alpha=0.1$}
\end{table}

We present our results in Table 1.  Since the non-conformity scores are continuous, we would expect tight concentration of the coverage around $0.9$, which we see for all models\footnote{It should be noted that there were 5 runs of scheme 2 that resulted in too small of a sample size for conformal prediction in the second wave; these cases were removed from our study. In any case, coverage would still be close to 0.9 if these cases were treated as ones where conformal prediction did not cover the truth.}. We also see that prediction intervals involving the linear model have  narrower width overall compared to their smoothing spline counterparts, which is also not surprising since we included more relevant covariates in the linear model.  It should be noted that the second wave of scheme 1 has a substantially larger number of nodes compared to other schemes since all neighbors are recruited in this scheme.  However, perhaps somewhat surprisingly, the average width associated with the second wave of scheme 1 is larger than the other schemes.  While it is not clear how this phenomenon generalizes to other settings, it appears that snowball sampling with fixed $r$ may lead to tighter prediction intervals than schemes involving all neighbors if members of a given wave are in some sense quite similar and there is substantially less variability within the referral group compared to a larger group of neighbors.  A similar effect is present when comparing the average widths for hop 2 to those of wave 2 for schemes 2 and 3.  While the second wave of scheme 3 appeared to produce the prediction intervals with the narrowest width in our simulation study, the tradeoff between breadth and depth for conformal prediction with snowball sampling merits further investigation.                    
\subsection{Random Walk}
We now consider a simulation study in which only 100 nodes are traversed by a random walk on a graph with 2000 nodes.  For the regression component of the model, let
\begin{align*}
Y = 3 + 8 X_{1} +4 \sin(4 \pi X_{2}) +3X_3  + \epsilon,
\end{align*}
where $\epsilon \sim N(0,1)$ and $(X_1,X_2, X_3) \sim N(\mu, \Sigma)$ with:

\begin{align*}
\mu = (1,3,0) , \ \
\Sigma = \begin{pmatrix}
1 & 0.6 & 0.3 \\ 
0.6 & 4 & \text{-} 0.4 \\ 
0.3 & \text{-} 0.4 & 1
\end{pmatrix}.
\end{align*}

For the network component of the model, we consider a sparsity level $\nu_n = 0.1$ and a variant of the Gaussian latent space model:
\begin{align*}
A_{ij} \sim \mathrm{Bernoulli}(\nu_n g(X_{3i}, X_{3j})),
\end{align*}
where:
\begin{align*}
g(x,y) = 0.9\exp(-(x-y)^2/4) +0.1.
\end{align*}
For the fitted model, we implement linear regression with $X_1$ and $X_2$ as the covariates; we treat $X_3$ as unobserved. Out of the 100 sampled nodes, 50 will be used to train a regression model.  We use the non-conformity score $|y-\hat{\mu}_n(x_1,x_2)|$ with $\alpha = 0.2$.  We compare the coverage and width of our weighted conformal prediction procedure to an analogous procedure constructed on 100 uniformly sampled nodes over 500 simulation runs.

\begin{table}[h]
\centering
\begin{tabular}{c|c|c|}
\cline{2-3}
                                     & Coverage & Width  \\ \hline
\multicolumn{1}{|c|}{Random walk}    & 0.838    & 10.996 \\ \hline
\multicolumn{1}{|c|}{Uniform sample} & 0.810    & 10.790 \\ \hline
\end{tabular}
\caption{Average coverage and width for random walk sampling computed over 500 runs for $\alpha=0.2$}
\end{table}

We present our results in Table 2.  As our theory suggests, the weighted conformal prediction procedure rapidly approaches the target coverage level and already has the desired coverage properties at $m=50$, at least in our one simulation example.  It turns out that the difference in width compared to the uniform sampling approach is significant at the $0.05$ level, but nonetheless it is reassuring that the random walk approach provides prediction sets with widths that are not alarmingly different even with a relatively small sample.

\subsection{Real Data Example: Predicting Twitch Views}
We now examine the performance of conformal prediction on the Twitch gamers dataset  \citep{rozemberczki2021twitch}, which consists of approximately 168,000 users from the live streaming service Twitch.  The users in the dataset come from the largest connected component of an undirected social network, which is also provided in the dataset.  We consider the task of constructing prediction intervals for views based on both conventional covariates and network covariates constructed from the social network.  Due to the massive size of the network, dealing with subnetworks is advantageous from a computational perspective.  We evaluate the performance of conformal prediction on both ego networks and random walks on the graph.

To quantify the performance of conformal prediction on ego networks, we randomly select 5 ego networks with at least 4000 neighbors.  For each ego dataset, we leave out 500 observations as a test set. The remaining observations are split into two equal-sized folds.  On the training set, we fit a random forest model where the conventional covariates are broadcasting language (binarized to English or not), presence of explicit content, affiliate status, and account lifetime. For the network covariates, we include degrees and a dimension 5 RDPG embedding, both computed on the ego network.  We again use the non-conformity score $|y - \hat{\mu}_n(x,z)|$ and set $\alpha =0.1$. 

\begin{table}[h]
\centering
\begin{tabular}{|c|c|c|}
\hline
Ego ID & Ego size & Coverage \\ \hline
10097  & 5003     & 0.904    \\ \hline
145505 & 5462     & 0.890    \\ \hline
160608 & 4843     & 0.906    \\ \hline
14649  & 4162     & 0.896    \\ \hline
143081 & 4998     & 0.918    \\ \hline
\end{tabular}
\caption{Results for conformal prediction on 5 ego networks with $\alpha=0.1$}
\end{table}

Our results for ego networks are presented in Table 3. Our results suggest that joint exchangeability is a reasonable assumption for this dataset since coverage on the test set is close the nominal level.  While our non-conformity score could use improvements in the sense that the lower bound of the prediction interval can be negative for smaller predicted view counts and the fitted model could be improved, our results here are quite encouraging regarding the applicability of conformal prediction to real-world ego networks.   

For the random walk, we consider $m=1000,2500,5000$.  For each $m$, we implement split conformal prediction with an equal split size of $m$.  For the fitted model, we consider both a linear model and a random forest, with the same covariate set as the ego network, but without the network covariates.  For the starting point of the random walk, we randomly choose one of the nodes with degree greater than 4000.  We study average coverage over 100 initializations of the random walk, where for each initialization we generate 500 independently drawn test points.  We use the non-conformity score $|y - \hat{\mu}_n(x,z)|$ and set $\alpha =0.1$.

We tabulate our results in Table 4.  It appears that our random walk approach under-covers for both models.  We believe that this largely has to do with the degree distribution of the network.  Roughly 2\% of all nodes in the dataset have only 1 neighbor.
Moreover, the degree distribution appears to be heavy-tailed, with a mean degree of 80.87, a median degree of 32, and a maximum degree of 35,279.  
Therefore, it is  questionable whether a lower and upper-bounded graphon is a reasonable approximation of the underlying data generating process.  It is quite possible that the mixing rate is substantially slower than what our theory based on these assumptions predicts.  While finite sample validity holds only within ego and snowball samples, the regularity conditions in these cases are much weaker than those for the random walk. Thus, it is not surprising that the former can have better coverage properties on the same real-world network compared to the random walk approach. 

\begin{table}[]
\centering
\begin{tabular}{cccc}
                                        & \multicolumn{3}{c}{Coverage}                                                             \\ \cline{2-4} 
\multicolumn{1}{c|}{}                   & \multicolumn{1}{c|}{m=1000} & \multicolumn{1}{c|}{m=2500} & \multicolumn{1}{c|}{m=5000} \\ \hline
\multicolumn{1}{|c|}{Linear Regression} & \multicolumn{1}{c|}{0.820}  & \multicolumn{1}{c|}{0.811}  & \multicolumn{1}{c|}{0.802}   \\ \hline
\multicolumn{1}{|c|}{Random Forest}     & \multicolumn{1}{c|}{0.818}  & \multicolumn{1}{c|}{0.800}  & \multicolumn{1}{c|}{0.784}   \\ \hline
\end{tabular}
\caption{Average coverage for weighted conformal prediction with data sampled by a random walk with $\alpha = 0.1$}
\end{table}

\section{Discussion}
In this work, we study the validity of conformal prediction under various sampling mechanisms.  We argue that, if the selection mechanism satisfies a certain invariance property, then conformal prediction retains its finite sample coverage guarantee within the (potentially) non-representative sample.  One crucial question is: how robust are these conditional coverage guarantees when the selection event deviates from the hypothetical selector? As a simple example, if one is interested in constructing a prediction set for the test point $Y_{n+1},$ one may examine the neighbors of the test point and then choose the largest ego network that contains the test point. This selection event is obviously different from what is proposed in the main paper, but is a very natural selection rule to consider in practice.  In addition, it would also be very interesting to study the properties of conformal prediction when the observed nodes come from other Markov processes on the graph, particularly in settings with hubs or disconnected components, which our current theory does not allow.  Finally, it should be noted that the framework developed in this paper applies not only to network-assisted regression problems, but also to link prediction problems as well.    
\section*{Acknowledgements}
The author would like to thank Elizaveta Levina, Ji Zhu, and Purnamrita Sarkar for helpful discussions. The author would also like to acknowledge support from Washington University's TRIADS seed grant program. 

\begin{appendix}
\section{Proofs for Section \ref{sec:conditonal-validity}}
\textit{Proof of Theorem 1.} \ \ 
After conditioning on the selection event $\mathcal{S}(V) = S$, we may consider the indices of $\mathcal{V}$ fixed.  Since the CDF characterizes the distribution, it suffices to consider a collection of hyperrectangles $\{\mathcal{B}_{ij}\}_{i,j \in S}$.  For an arbitrary collection of hyperrectangles, we may write the conditional probability after permutation by $\sigma \in \Sigma_S$ as:  
\begin{align*}
P\left( \bigcap_{i,j \in S} \mathcal{V}_{\sigma(i) \sigma(j)} \in \mathcal{B}_{ij}  \ \biggr\rvert \  \mathcal{S}(V) = S \right) &=  \frac{P\left( \bigcap_{i,j \in S } V_{\sigma(i) \sigma(j)} \in \mathcal{B}_{ij} \  \cap \ \{\mathcal{S}(V) = S  \} \right)}{P(\mathcal{S}(V) = S )}
\\ &= \frac{P\left( \bigcap_{i,j \in S} V_{\sigma(i) \sigma(j)} \in \mathcal{B}_{ij} \  \cap \ \{\mathcal{S}(V^\sigma) = S  \} \right)}{P(\mathcal{S}(V) = S )}
\\ &= \frac{P\left(\bigcap_{\sigma(i),\sigma(j) \in S} V_{ij} \in \mathcal{B}_{ij} \  \cap \ \left\{V^\sigma \in \mathcal{S}^{-1}\{S\} \right\}  \right)}{P(\mathcal{S}(V) = S )}
\\ &= \frac{P\left(\bigcap_{i,j \in S} V_{ij} \in \mathcal{B}_{ij} \  \cap \ \{\mathcal{S}(V) = S  \} \right)}{P(\mathcal{S}(V) = S )}
\\ &= P\left(\bigcap_{i,j \in S}  \mathcal{V}_{ij} \in \mathcal{B}_{ij}  \ \biggr\rvert \  \mathcal{S}(V) = S \right).
\end{align*}
The result follows. \qed \\ 

To prove Theorem \ref{theorem-conditional-conformal-theorem}, we now recall the following result from \citep{network-conformal}, which is an adaptation of \citep{kuchibhotla2021exchangeability,commenges-transformations,dean-verducci}:
\begin{proposition}
\label{theorem-exchangeability-transformations}
 Let $X$ be a random variable taking values in $\mathcal{X}$ and suppose that $Y = H(X)$ for some $H: \mathcal{X} \mapsto \mathcal{Y}$.  Further suppose that for some collection of functions $\mathcal{F}$ $\mathcal{X} \mapsto \mathcal{X}$,
 \begin{align*}
 \label{eq-invariance-assumption}
  f(X) \stackrel{d}{=} X \ \ \forall \ f \in \mathcal{F}.  
 \end{align*}
Furthermore, let $\mathcal{G}$ be a collection of functions $\mathcal{Y} \mapsto \mathcal{Y}$ and suppose that for any $g \in \mathcal{G}$, there exists a corresponding $f \in \mathcal{F}$ such that,
\begin{align*}
g(H(X)) = H(f(X)) \ \  a.s.    
\end{align*}
Then,
\begin{align*}
g(Y) \stackrel{d}{=} Y \ \ \forall \ g \in \mathcal{G}.   
\end{align*}
\end{proposition}

\noindent \textit{Proof of Theorem \ref{theorem-conditional-conformal-theorem}.} \ \ 
We begin by arguing that $(Y_i, X_i, \hat{Z}_i)_{i \in \mathcal{S}}$ is exchangeable conditional on $\mathcal{S} = S$ for each $S$ such that $P(\mathcal{S} = S) > 0$. By Proposition \ref{theorem-exchangeability-transformations}, it suffices to show that for $H$ that maps $\mathcal{V}$ to $(Y_i, X_i, \hat{Z}_i)_{i \in \mathcal{S}}$, permuting the rows and columns of $\mathcal{V}$ corresponds to permuting $(Y_i, X_i, \hat{Z}_i)_{i \in \mathcal{S}}$ accordingly.  This is guaranteed by Assumption \ref{assumption-permutation-invariance}.  Now, we may consider permutations acting only on $D_2$ and invoking Proposition \ref{theorem-exchangeability-transformations} again, we have that $(S_i)_{i \in D_2}$ is exchangeable.  Finally, (\ref{eq-conditional-coverage-guarantee}) follows from properties of quantiles of exchangeable random variables (see for example, Lemma 1 of \citet{NEURIPS2019_8fb21ee7}).  \qed \\

\noindent \textit{Proof of Proposition \ref{proposition-referral-exchangeable}.}  \ \  
Since it is clear that $(W_{ij})_{1 \leq i,j, \leq n}$ is jointly exchangeable, by Proposition \ref{theorem-exchangeability-transformations}, it suffices to show that, for any permutation $\sigma$, permuting the rows and columns of $W$ and then applying the row-wise rank function corresponds to permuting the rows and columns of $A$ accordingly.

Let $j_1, \ldots, j_r$ correspond to the column indices of the $r$ largest values of $W_{i1}, \ldots, W_{in}.$  Applying $\sigma$, we see that $\sigma(j_1), \ldots , \sigma(j_r)$ must be the indices of the $r$ largest values of $(W_{\sigma(i) \sigma(1)}, \ldots, W_{\sigma(i)\sigma(n)}).$  Thus, permuting the rows and columns of $W$ and then applying the row-wise rank function corresponds to permuting the rows and columns of $A$. The result follows. \qed 
\\ \\
\noindent \textit{Proof of Proposition \ref{proposition-snowball-invariant}.} \ \ For notational simplicity, we will view the selection rule as a function of $A$.  For $a = (a_{ij})_{1 \leq i,j \leq n}$, it suffices to show that:
\begin{equation}
\label{eq-wave-to-show}
\{ a \ | \ a \in \{ \mathcal{W}(M_0, k) = S \} \} = \{ a \ | \ a^\sigma \in \{ \mathcal{W}(M_0, k) = S \}\}.
\end{equation}
We will first show $\{ a \in \{\mathcal{W}(M_0, k) = S\} \} \subseteq \{ a^\sigma \in \{\mathcal{W}(M_0, k) = S \} \}$.  To this end, pick an arbitrary configuration $a \in \{ a \in \{\mathcal{W}(M_0, k) = S\} \}$. Of course, if the set is empty, we do not need to verify the property, so we consider the case in which at least one such element exists. We will argue that it must also be the case that $a \in \{ a^\sigma \in  \{ \mathcal{W}(M_0, k) = S\} \}$.
First, let:
\begin{align*}
\mathcal{B}_k =  \left(\bigcup_{0 \leq t \leq k-1}\mathcal{W}(M_0,t)\right)^c.
\end{align*}
 Notice that when applying any permutation in $\Sigma_S$ to the rows and columns of $a$, edges determining membership in previous waves are unchanged by construction. Thus, the potential candidates for the $k$th wave, given by $\mathcal{B}_k$, remains unchanged by permutation.
 Now, applying $\sigma$, we see that, since $ a \in \{\mathcal{W}(M_0, k) = S \}$ and $\sigma$ restricted to $S$ is a bijection:
 \begin{align*}
 \bigcup_{i \in \mathcal{W}(M_0,k-1) } \{j \in \mathcal{B}_k \ | \  A_{ij}^\sigma =1 \} = S.
 \end{align*}
 Thus the first inclusion is proved. To see that $\{  a^\sigma \in \{\mathcal{W}(M_0, k) = S\} \}  \subseteq  \{ a \in \{\mathcal{W}(M_0, k) = S\} \}$, we can again pick an arbitrary element $a \in \{ a^\sigma \in \{\mathcal{W}(M_0, k) = S\} \}$ and then argue that $a^{\sigma'} \in \{\mathcal{W}(M_0, k) = S \}$ using the same reasoning, where $\sigma' = \sigma^{-1}$.  Since $\sigma \in \Sigma_S$ and $\Sigma_S$ were arbitrary, the result follows.
 \qed
\\ \\ 
\noindent \textit{Proof of Proposition \ref{proposition-k-hop-invariant}.} \ \ 
Let $d(u,v)$ denote the shortest path (geodisic) distance\footnote{For directed graphs, $d$ is a quasi-metric since $d(u,v)$ need not equal $d(v,u)$. } between nodes $u$ and $v$. Furthermore, let $d^\sigma(u,v)$ denote the distance between $u$ and $v$ for the relabeled adjacency matrix $A^\sigma$.    To show the analogue of (\ref{eq-wave-to-show}), it suffices to establish that, for any $S$ that does not contain elements belonging to $M_0$ and any $\sigma \in \Sigma_S$,
\begin{align*}
& \min_{i \in M_0} d(i,j) \leq k \ \ \forall \ j \in S \ \text{ and }  \ \min_{i \in M_0} d(i,j) > k \ \ \forall \ j \in S^c \setminus M_0
\\ \iff & \min_{i \in M_0} d^\sigma(i,j) \leq k \ \ \forall \ j \in S \ \text{ and }  \ \min_{i \in M_0} d^\sigma(i,j) > k \ \ \forall \ j \in S^c \setminus M_0.   
\end{align*}
We will start by showing the $\implies$ direction.  Pick $j \in S$ for the permuted graph. Since $\sigma$ restricted to $S$ is a bijection, it follows that $j' = \sigma(j)$ for some $j' \in S$.  Now, by assumption, $\min_{i \in M_0} d(i,j') \leq k$; thus $\min_{i \in M_0} d^\sigma(i,j) \leq k$  follows since the length of the shortest path is unchanged by relabeling vertices.  For $j \in S^c \setminus M_0$, we have that $\sigma(j) = j'$ and again, since the length of the shortest path is unaffected, we have $\min_{i \in M_0} d^\sigma(i,j) > k$.      

For $\impliedby$, pick $j \in S$ from the original graph.  We see that $j =\sigma(j') \in S$ for some $j' \in S$ and we can repeat the same reasoning as above. The case where $j \in S^c \setminus M_0$ is analogous. The result follows.   
\qed



\section{Properties of $\beta$-mixing Coefficients}
To prove Theorem \ref{thm-conformal-prediction-random-walk}, we need to characterize the dependence properties of the random walk. The notion of mixing provides a convenient framework; see for example, \citet{10.1214/154957805100000104} or \citet{doukhan-mixing} for an overview on the relationship between Markov chains and various notions of mixing.  In particular, we focus on the notion of $\beta$-mixing or absolute regularity to invoke an empirical process result due to \citet{10.1214/aop/1176988849}.  We start with the notion of $\beta$-dependence, which may then be used to construct a corresponding dependence coefficient.   In what follows let $\mathcal{A}$ and $\mathcal{B}$ be $\sigma$-fields.   
\begin{definition}[$\beta$-dependence measure]
\begin{align*}
\beta(\mathcal{A}, \mathcal{B})  = \frac{1}{2} \sup \sum_{i=1}^I \sum_{j=1}^J |P(A_i \cap B_j) - P(A_i) P(B_j)|  
\end{align*}
where the supremum is taken over all finite partitions in $\mathcal{A}$ and $\mathcal{B}$. 
\end{definition}
Now, we define the $\beta$-mixing coefficient, which measures the decay in dependence as a function of gap between the past and the future.  While $\beta$-mixing coefficients are more commonly defined for a two-sided stochastic process, the following finite form will be sufficient for our purposes.  

\begin{definition}[$\beta$-mixing coefficient]
\begin{align*}
\beta_t(U_0^n) = \sup_{0 \leq \tau \leq n-t  } \beta(\sigma(U_0, \ldots, U_{\tau}), \sigma(U_{\tau+t}, \ldots, U_n)) 
\end{align*}
\end{definition}
For a stationary Markov chain, we have:
\begin{equation}
\label{eq-markov-chain-dependence-bound}
\beta_t(U_0^n) \leq  \sup_{i} \|P^t(\cdot, i) - \pi(\cdot)    \|_{TV}.
\end{equation}
See for example, \citet{meitz_saikkonen_2021} or \citet{https://doi.org/10.1111/j.1467-9892.2005.00412.x}.


\section{Proofs for Section \ref{sec:asymptotic-validity-random-walk}}
We start by proving a lemma that establishes that, under appropriate conditions, the weighted empirical CDF of the non-conformity scores computed on the random walk will, with high probability, be close to the empirical CDF of the non-conformity scores computed on all $n$ nodes.

We now prepare some additional notation.  In what follows, let $\tilde{S}_i = s(\tilde{Y}_i,\tilde{X}_i;\widetilde{\mathcal{D}}_1)$, $\tilde{S}_i^* = s(\tilde{Y}_i,\tilde{X}_i)$, $\tilde{S}' =s(Y',X';\widetilde{\mathcal{D}}_1)$, $S' =s_*(Y',X')$,   and $S_i = s_*(Y_i, X_i)$.  Furthermore, for notational convenience, let $\tilde{\nu}_i = \nu(\mathcal{X}_{m+i})$ and $\nu_i = \nu(i)$.  

\begin{lemma}
\label{lemma-cdf-approximation-weighted}
Under the conditions of Theorem \ref{thm-conformal-prediction-random-walk}, 
\begin{align*}
& \left| \frac{1}{m}\sum_{i=1}^m \tilde{\nu}_i \mathbbm{1}(\tilde{S}_i \leq \tilde{S}') - \frac{1}{n}\sum_{j=1}^n \mathbbm{1}(S_j \leq S') \right| = o_P(1).
\end{align*}
\end{lemma}
\begin{proof}
We derive an upper bound; the lower bound follows analogous reasoning. On the high probability set $\{|\tilde{S}' - S'| \leq \delta \}$, we have that:
\begin{align*}
\frac{1}{m}\sum_{i=1}^m \tilde{\nu}_i \mathbbm{1}(\tilde{S}_i \leq \tilde{S}') \leq \frac{1}{m}\sum_{i=1}^m \tilde{\nu}_i \mathbbm{1}(\tilde{S}_i^* \leq S'+2\delta) +\frac{1}{m}\sum_{i=1}^m \nu_{max} \mathbbm{1}(|\tilde{S}_i- \tilde{S}_i^*| > \delta),  
\end{align*}
where $\nu_{max} = \max_{1 \leq i \leq n} \nu_i$. Now, consider the bound:
\begin{align*}
& \ \frac{1}{m}\sum_{i=1}^m \tilde{\nu}_i \mathbbm{1}(\tilde{S}_i \leq \tilde{S}') - \frac{1}{n} \sum_{j=1}^n \mathbbm{1}(S_i \leq S')
\\ \leq & \ \sup_{t \in \mathbb{R}} \left| \frac{1}{m}\sum_{i=1}^m \tilde{\nu}_i \mathbbm{1}(\tilde{S}_i^* \leq t + 2\delta)
- \frac{1}{n} \sum_{j=1}^n \mathbbm{1}(S_j \leq t + 2\delta) \right|
\\ \ & \ + \sup_{t \in \mathbb{R}} \left| \frac{1}{n} \sum_{j=1}^n [\mathbbm{1}(S_j \leq t + 2\delta) - P(S_j \leq t + 2\delta )]\right|
\\ \ & \ + \sup_{t \in \mathbb{R}} \left| P(S' \leq t + 2\delta) - P(S' \leq t) \right| 
\\ \ & \ + \sup_{t \in \mathbb{R}} \left|  \frac{1}{n} \sum_{j=1}^n [\mathbbm{1}(S_j \leq t) - P(S_j \leq t)] \right| + \frac{1}{m}\sum_{i=1}^m \nu_{max} \mathbbm{1}(|\tilde{S}_i^*- \tilde{S}_i| > \delta)
\\ = & \ I + II + III + IV + V, \text{ say. }
\end{align*}
Now for some $\mathcal{C}_n$ such that $P(\mathcal{C}_n) \rightarrow 1$, we will show that:
\begin{align*}
\sup_{\omega \in \mathcal{C}_n} P\left( \sup_{t' \in \mathbb{R}} \left| \frac{1}{m}\sum_{i=1}^m \tilde{\nu}_i \mathbbm{1}(\tilde{S}_i^* \leq t') - \frac{1}{n} \sum_{j=1}^n \mathbbm{1}(S_j \leq t')  \right| > \epsilon \ \biggr\rvert \ V, \mathcal{X}_0 \right) \rightarrow 0. 
\end{align*}
By Theorem \ref{thm:mixing-time} and Lemma \ref{lemma-max-stationary-measure}, it follows that there exists $0 < K< \infty$ and $0 \leq \gamma <1$ and $ 0 < M <\infty$ such that $P(\mathcal{C}_n) \rightarrow 1$, where $\mathcal{C}_n$ is the set given by:
\begin{align*}
\mathcal{C}_n = \left\{ \forall \ t \geq 1, \ \max_{x_0 \in \{1, \ldots, n\}} \|\widetilde{P}_n^t(\ \cdot \ | \ \mathcal{X}_0 = x_0) - \pi(\cdot) \|_{TV} \leq K \sqrt{n} \gamma^t    \right\} \cap \left\{ \max_{1 \leq i \leq n} \frac{1}{n\pi(i)} \leq M \right\}
\end{align*}
Let $\mathcal{X}_0, \ldots, \mathcal{X}_{2m}$ denote the sequence of nodes drawn from the random walk. For each $\omega \in \mathcal{C}_n$, define the function:
$\phi: \{1, \dots, n\} \mapsto \{(\nu_1, S_1), \ldots, (\nu_n, S_n) \}$ as $\phi(i) = (\nu_i, S_i)$ and let $\tau_j = \phi(\mathcal{X}_j)$.  Furthermore, consider the following induced probability measure associated with $\tau$:
$\tilde{P}_n(\tau_j \in \cdot \ | \ \mathcal{X}_0 = x_0)$ and a corresponding measure initialized  by the stationary measure $\tilde{P}_n^{\pi}.$  
Since $\|P-Q\|_{TV} = \int_{\|f\|_\infty \leq 1} f \ d(P-Q)$ and the induced measure also satisfies a Markov property conditional on $\tau_m$ since $S_i$ are distinct almost surely, it is clear that, uniformly on the set $\mathcal{C}_n$:
\begin{equation}
\label{eq:dist-to-stationary}
\begin{split}
& \ \|\tilde{P}_n( (\tau_{m+1}, \ldots, \tau_{2m}) \in \cdot \ | \ \mathcal{X}_0 = x_0) - \tilde{P}_n^\pi( (\tau_{m+1}, \ldots, \tau_{2m}) \in \cdot \ ) \|_{TV}
\\ \leq & \  \|\tilde{P}_n(  \tau_{m+1} \in \cdot \ | \ \mathcal{X}_0 = x_0) - \tilde{P}_n^\pi( \tau_{m+1} \in \cdot \ ) \|_{TV}
\\ = & \ \sup_{\|f\|_\infty \leq 1} \int f(\phi(\mathcal{X}_m)) \  d[\tilde{P}_n^{m+1}( \cdot \ | \  \mathcal{X}_0 = x_0) - \pi(\cdot)]
\\  \leq & \ K \sqrt{n}\gamma^{m+1}.
\end{split}
\end{equation}

Moreover, uniformly on the set $\mathcal{C}_n$, by similar reasoning the $\beta$-mixing coefficient for random walk initialized with the stationary distribution satisfies: 
\begin{align*}
\ & \ \beta_t(\tau_0, \ldots, \tau_{2m})
\\ = & \ \beta_t(\phi(\mathcal{X}_0), \ldots, \phi(\mathcal{X}_{2m}))
\\ \leq & \ \beta_t(\mathcal{X}_0, \ldots, \mathcal{X}_{2m})
\\ \leq & \   K \sqrt{n}\gamma^t.
\end{align*}

We first approximate the asymptotically stationary Markov chain with the stationary measure.  For notational simplicity, in what follows we use the notation $P^\pi(\cdot)$ below.  By (\ref{eq:dist-to-stationary}), since there is a gap of length $m+1$ from $\mathcal{X}_0$, we have, uniformly in $\mathcal{C}_n,$ 
\begin{align*}
& \sup_{\omega \in \mathcal{C}_n} P\left( \sup_{t' \in \mathbb{R}} \left| \frac{1}{m}\sum_{i=1}^m \tilde{\nu}_i \mathbbm{1}(\tilde{S}_i^* \leq t') - \frac{1}{n} \sum_{j=1}^n \mathbbm{1}(S_j \leq t')  \right| > \epsilon \ \biggr\rvert \ V, \mathcal{X}_0 \right)
\\  \leq &  \sup_{\omega \in \mathcal{C}_n} P^\pi\left( \sup_{t' \in \mathbb{R}} \left| \frac{1}{m}\sum_{i=1}^m \tilde{\nu}_i \mathbbm{1}(\tilde{S}_i^* \leq t') - \frac{1}{n} \sum_{j=1}^n \mathbbm{1}(S_j \leq t')  \right| > \epsilon \ \biggr\rvert \ V \right) + K \sqrt{n} \gamma^{m+1}.
\end{align*}
Now that we have a stationary process, we will invoke an empirical process result for stationary $\beta$-mixing processes.  To this end, let $g(v) = v \mathbbm{1}(v \leq M) + M \mathbbm{1}(v > M)$ and define:
\begin{align*}
\mathcal{F} &= \left\{ \mathbbm{1}(x \leq t) \ | \ t \in \mathbb{R}  \right\}
\\ \mathcal{G} &= \left\{ f \cdot g \ | \  f \in \mathcal{F}   \right\}. 
\end{align*}
By well-known VC-subgraph class preservation properties (see, for example,  \citet{van-der-Vaart-Wellner-weak-conv} Lemma 2.6.18), the class $\mathcal{G}$ is a VC-subgraph class.  Moreover, the envelope is bounded, while $g$ is equal to $\nu_i$ on the set $\mathcal{C}_n$.  With the chosen weights, it is also clear that conditional on the graph, the expectation under $P^\pi$ is $\frac{1}{n} \sum_{i=1}^n \mathbbm{1}(S_i \leq t')$.  An inspection of the proof of Theorem 3.1 of  \citet{10.1214/aop/1176988849} reveals that the rate of convergence is uniform over processes upper-bounded with the same $\beta$-coefficient; thus, this term goes to zero under the condition $m = \omega(\log n)$.
Since $P(\mathcal{C}_n) \rightarrow 1$, $I \stackrel{P}{\rightarrow} 0$.  

Now, $II$ and $IV$ converge in probability to $0$ by, for example, the Dvoretzky-Kiefer-Wolfowitz inequality, $III$ can be made arbitrarily small by Assumption A2 and $V \stackrel{P}{\rightarrow} 0 $ by Assumption A1, Lemma \ref{lemma-max-stationary-measure}, and Markov inequality.
\end{proof}

\noindent \textit{Proof of Theorem \ref{thm-conformal-prediction-random-walk}.} \ \ Given Lemma \ref{lemma-cdf-approximation-weighted}, the proof of Theorem \ref{thm-conformal-prediction-random-walk} is immediate since $(S_1, \ldots, S_n, S')$ are i.i.d. and continuous.  We demonstrate the lower bound; the upper bound is analogous.  For any $\epsilon$, define the event:
\begin{align*}
\mathcal{E} = \left\{ \left| \frac{1}{m}\sum_{i=1}^m \tilde{\nu}_i \mathbbm{1}(\tilde{S}_i \leq \tilde{S}') - \frac{1}{n}\sum_{j=1}^n \mathbbm{1}(S_j \leq S') \right| \leq \epsilon \right\}. 
\end{align*}
Observe that, for $n$ large enough, 
\begin{align*}
P(Y' \in \widetilde{C}_m(X')) - (1-\alpha) &=  P\left(\frac{1}{m}\sum_{i=1}^m \tilde{\nu}_i \mathbbm{1}(\tilde{S}_i \leq S') \leq 1-\alpha  \right) - (1-\alpha) 
\\ & \geq P\left(\left\{\frac{1}{m}\sum_{i=1}^m \tilde{\nu}_i \mathbbm{1}(\tilde{S}_i \leq S') \leq 1-\alpha \right\} \ \cap \  \mathcal{E}  \right) - (1-\alpha)
\\ & \geq P\left(\frac{1}{n}\sum_{j=1}^n \mathbbm{1}(S_j \leq S') \leq 1-\alpha-\epsilon \right) - (1-\alpha) - \epsilon
\\ & \geq - 3\epsilon.
\end{align*}
\qed 
\begin{remark}
The above argument also works for the case where $(Y',X')$ is independently drawn from $(Y_1, X_1), \ldots, (Y_n, X_n)$ rather than being a new independent draw from the underlying distribution. 
\end{remark}

We now establish the closeness of the eigenvalues of $\mathscr{A}$ to the eigenvalues of the operator $T_w$.    To quantify the similarity of spectra, we consider a rearrangment distance similar to what is considered in \citet{koltchinksii-gine-kernel-operator}.  First, recall that the spectrum of a compact, symmetric operator is countable. Thus, we may view the spectrum as a sequence.  For finite-dimensional operators, we append an infinite number of zeros so that its spectrum is also an (infinite) sequence.  We say that two sequences $(x_i)_{i \in \mathbb{N}}$ and $(y_i)_{i \in \mathbb{N}}$  are equivalent if one is a re-arrangement of the other.  That is, $(x_i)_{i \in \mathbb{N}}$ and $(y_i)_{i \in \mathbb{N}}$ are equivalent  if and only if there exists some permutation $\sigma: \mathbb{N} \mapsto \mathbb{N}$ such that $ x_{\sigma(i)} = y_i$ for all $i \in \mathbb{N}$. 
 On this quotient space, consider the distance:
 \begin{align*}
 \delta_\infty(x,y) = \inf_{\sigma} \sup_{i} |x_i - y_{\sigma(i)}|.
 \end{align*}
\citet{koltchinksii-gine-kernel-operator} consider $\delta_2$, but here we consider the weaker norm $\delta_\infty$ since the tools that we use are not strong enough to bound eigenvalue perturbations under mild conditions for the $\delta_2$ norm.  In any case, we do not need to uniformly control all eigenvalues for our purposes, so the $\delta_\infty$ norm is more than suitable.  In what follows let $\lambda(\cdot)$ denote the spectrum associated with an operator, again appended with zeros if necessary.     

\begin{lemma}
\label{thm-eigenvalue-convergence}
Suppose that $0 < c \leq w(x,y) \leq C < \infty$ and $\rho_n = \omega(\log^{3/2} n / n)$. Then,
\begin{align*}
\delta_\infty(\lambda(\mathscr{A}), \lambda(T_w)) \stackrel{P}{\rightarrow} 0.
\end{align*}
\end{lemma}

\begin{proof}
Recall that $\mathfrak{D}(x) = \int_0^1 w(x,y) \ dy$ and let $\mathscr{D}$ be a diagonal matrix such that $\mathscr{D}_{ii} = \mathfrak{D}(\xi_i)$.  Moreover, let $\mathscr{P} = (\mathscr{D}^{-1/2}P\mathscr{D}^{-1/2}),$ where $P_{ij} = w(\xi_i, \xi_j)$ for $i\neq j$ and $P_{ii} = 0$.  Let $n' = n-1$ and consider the following decomposition:
\begin{align*}
\delta_\infty(\lambda(\mathscr{A}), \lambda(T_w))  \leq  \delta_\infty(\lambda(\mathscr{A}),\lambda(\mathscr{P})/n') + \delta_\infty(\lambda(\mathscr{P})/n',\lambda(T_w)).
\end{align*}
For the second term, we see that for $i \neq j$:
\begin{align*}
\mathscr{P}_{ij} = \frac{w(\xi_i, \xi_j)}{\sqrt{\mathfrak{D}(\xi_i) \mathfrak{D}(\xi_j})} = K(\xi_i,\xi_j) \text{, say. }
\end{align*}
Under the conditions $c \leq w(x,y) \leq C$, $EK^2(X,y) <\infty$; moreover, it is symmetric.  Therefore, convergence of this second term in probability to zero follows from Theorem 3.1 of \citet{koltchinksii-gine-kernel-operator}. 

Now for the first term, we may further upper bound with:
\begin{align*}
\delta_\infty(\lambda(\mathscr{A}), \lambda(\mathscr{P})/n') &\leq  \delta_\infty(\lambda(\mathscr{A}), \lambda(\mathscr{D}^{-1/2}A\mathscr{D}^{-1/2})/(n'\rho_n)) 
\\ & \ \ + 
\delta_\infty(\lambda(\mathscr{D}^{-1/2}A\mathscr{D}^{-1/2})/(n' \rho_n), \lambda(\mathscr{P})/n') 
\\ &= I + II \text{, say. }
\end{align*}
where $n' = n-1$. By Lemma \ref{lemma-perturb-degrees}, $I \stackrel{P}{\rightarrow} 0.$  Furthermore, by Lemma \ref{lemma-eigenvalue-bound}, $II \stackrel{P}{\rightarrow} 0.$  The result follows.
\end{proof}
In what follows, recall that $n' = n-1.$  Note that statements similar to Lemmas \ref{lemma-perturb-degrees} and \ref{lemma-eigenvalue-bound} are more well-known, but we provide proofs here for completeness. 
\begin{lemma}
\label{lemma-perturb-degrees}
Under the conditions of Theorem \ref{thm-eigenvalue-convergence}, 
\begin{align*}
\delta_\infty(\lambda(\mathscr{A}), \lambda(\mathscr{D}^{-1/2}A\mathscr{D}^{-1/2})/(n' \rho_n)) \stackrel{P}{\rightarrow} 0.
\end{align*}
\end{lemma}

\begin{proof}
Let $R$ be the matrix where:
\begin{align*}
R_{ij} &= \frac{A_{ij}\left(n'\rho_n \sqrt{\mathfrak{D}(\xi_i)\mathfrak{D}(\xi_j)} - \sqrt{D_i D_j}\right) }{n'\rho_n \sqrt{ \mathfrak{D}(\xi_i)\mathfrak{D}(\xi_j) D_i D_j}}.
\end{align*}
It follows that, uniformly in $1 \leq i \leq n,$
 \begin{align*}
|\lambda_i(\mathscr{A})- \lambda_i(\mathscr{D}^{-1/2}A\mathscr{D}^{-1/2})/(n\rho_n)| & \stackrel{(i)}{\leq} \|R \|_{op} 
\\ & \stackrel{(ii)}{\leq} \max_{1 \leq i \leq n} \sum_{j=1}^n |R_{ij}|. 
\end{align*}
where $(i)$ follows from Weyl's inequality and $(ii)$ follows from the fact for any matrix $m \times n$ $B$, $\|B\|_{op}^2 \leq \| B\|_1 \|B\|_\infty$ where $\|B\|_1$ and $\|B\|_\infty$ are defined as:
\begin{align*}
\|B\|_1 &= \sup_{v \neq 0} \frac{\|Bv\|_1} {\|v\|_1} = \max_{1 \leq j \leq n} \sum_{i=1}^m |B_{ij}|
\\ \|B\|_\infty &= \sup_{v \neq 0} \frac{\|Bv\|_\infty}{\|v\|_\infty} =  \max_{1 \leq i \leq m} \sum_{j=1}^n |B_{ij}|.
\end{align*}
Note that if $B$ is symmetric, $\|B\|_\infty = \|B\|_1.$ Now, by mean value theorem and $\mathfrak{D}(\xi) > c$, for $ \eta_i, \eta_j$ between $D_i/(n'\rho_n)$ and $\mathfrak{D}(\xi_i)$, $D_j/(n'\rho_n)$ and $\mathfrak{D}(\xi_j)$, respectively:  
\begin{align*}
\ & \ \max_{1 \leq i \leq n} \sum_{j=1}^n |R_{ij}|
\\ \leq \ &  \frac{\max_{i} D_i}{ c\min_i D_i} \max_{i,j} \left|\frac{\eta_i + \eta_j}{\sqrt{\eta_i \eta_j}} \right| \times   \max_{i} \left| \frac{D_i}{n'\rho_n} - \mathfrak{D}(\xi_i) \right| 
\\  = \ &  I \times II \times III, \text{ say.}
\end{align*}
Now, observe that $D_i \ | \ \xi_i \sim \mathrm{Binomial}(n', \rho_n \mathfrak{D}(\xi_i))$.  Thus, by a Chernoff bound, it follows that $\max_i | D_i - n'\rho_n\mathfrak{D}(\xi_i)| \leq K \sqrt{n \rho_n \log n}$ with probability tending to 1 for some $0 < K < \infty$ since if we choose $\epsilon_n = K \sqrt{ n \rho_n \log n}$, we have: 
\begin{align*}
& P\left( \max_{i} D_i - n'\rho_n \mathfrak{D}(\xi_i) > \epsilon_n \right)
\\ \leq \ & n \  \mathbb{E}\left[P\left( D_i - n'\rho_n \mathfrak{D}(\xi_i) > \epsilon \ | \ \xi_i\right)\right]
\\ \leq \ &  n \exp\left(- \frac{\epsilon_n^2}{3Cn'\rho_n} \right) = o(1).  
\end{align*}
Since the same Chernoff bound holds for the lower tail, we have for some $0 <c'\leq C' < \infty$, with probability tending to 1:
\begin{equation}
\label{eq-min-max-degree}
\begin{split}
\max_i D_i &\leq Cn'\rho_n + K\sqrt{n \rho_n \log n} \leq C'n \rho_n 
\\ \min_i D_i &\geq cn'\rho_n - K\sqrt{n \rho_n \log n} \geq c'n \rho_n.
\end{split}
\end{equation}
Thus, it can readily be seen that $I = O_P(1)$, $II = O_P(1)$, and $III = o_P(1).$  The result follows. 
\end{proof}

\begin{lemma}
\label{lemma-eigenvalue-bound}
Under the conditions of Theorem \ref{thm-eigenvalue-convergence}, 
\begin{align*}
\delta_\infty(\lambda(\mathscr{D}^{-1/2}A\mathscr{D}^{-1/2})/(n'\rho_n)- \lambda(\mathscr{P})/n') \stackrel{P}{\rightarrow} 0. 
\end{align*}
\end{lemma}

\begin{proof}
We begin by deriving a bound for $\|H\|_{op},$ where:
\begin{align*}
H = \mathscr{D}^{-1/2}A\mathscr{D}^{-1/2} - \rho_n\mathscr{P}.
\end{align*}
It is clear that conditional on $\boldsymbol{\xi}_n = (\xi_1, \ldots, \xi_n)$, $H$ is centered with independent entries.  Moreover, observe that:
\begin{align*}
|H_{ij}| \leq \frac{1}{c}, \ \ \sigma^2 \leq \frac{C}{c^2}\rho_n 
\end{align*}
Thus, invoking Theorem 1.4 of \citet{Vu-Spectral-Norm}, we have for some $0 < K < \infty$,
\begin{align*}
\|H\|_{op} \leq K \sqrt{n \rho_n } \log^{3/4} n.
\end{align*}
Now, Weyl's inequality implies that, uniformly in $i$:
\begin{align*}
\{\lambda_i(\mathscr{D}^{-1/2}A\mathscr{D}^{-1/2})/(n'\rho_n)- \lambda_i(\mathscr{P})/n'\} \leq \sqrt{\frac{K'\log^{3/2} n}{n \rho_n}}.
\end{align*}
The result follows. 

\end{proof}
\begin{lemma}
\label{lemma-max-stationary-measure}
Suppose that $\rho_n = \omega(\log n/n)$ and $0 < c \leq w(u,v) \leq C < \infty$. Then, there exists $0<M<\infty$ such that:
\begin{align*}
P\left( \max_{1 \leq i \leq n} \frac{1}{n\pi(i)} \leq M \right) \rightarrow 1.
\end{align*}
\begin{proof}
Recall that the stationary measure is given by:
\begin{align*}
\pi(i) = \frac{D_i}{2|E|} = \frac{D_i}{n\rho_n(n-1)} \times \left( \frac{1}{ {n \choose 2} \rho_n } \sum_{i < j} A_{ij} \right)^{-1} = I \times II, \text{ say.}
\end{align*}
By Theorem 1 of \citet{Bickel-Chen-Levina-method-of-moments}, $II \stackrel{P}{\rightarrow} 1$.  Moreover, from (\ref{eq-min-max-degree}), it is clear that there exists some $k >0$ such that, with probability tending to 1:
\begin{align*}
\min_{1 \leq i \leq n} \frac{D_i}{n'\rho_n} \geq k.
\end{align*}
The result follows.
\end{proof}

\end{lemma}

\noindent \textit{Proof of Theorem \ref{thm:mixing-time}}. \ \ 
Recall that:
\begin{align*}
\|\widetilde{P}_n^t(\ \cdot \ | \ \mathcal{X}_0 = i) - \pi(\cdot) \|_{TV} \leq \frac{ \gamma(\mathscr{A})^t}{\sqrt{\pi(i)}}.
\end{align*}
By Lemma \ref{lemma-max-stationary-measure}, it must be the case that, for some $0 < M < \infty$, with probability tending to $1$:
\begin{align*}
\max_{1 \leq i \leq n} \frac{1}{\sqrt{\pi(i)}} \leq \sqrt{Mn}.
\end{align*}

Now by Lemma \ref{thm-eigenvalue-convergence}, it must be that case that $\lambda_{2}(\mathscr{A}) \stackrel{P}{\rightarrow} \lambda_{2}(T_w)$ and $\lambda_{-1}(\mathscr{A}) \stackrel{P}{\rightarrow} \lambda_{-1}(T_w)$. Furthermore, $\lambda_1(T_w) =1$ since $\lambda_1(\mathscr{A}) =1$ for all $n$.  Moreover, recall the Krein-Rutman theorem, stated below (see for example, \citet{chang-positive-operators-pdf})

\begin{proposition}[Krein-Rutman theorem, strong version] 
\label{proposition-krein-rutman}
Let $X$ be a Banach space, $P \subset X$ be a total cone and $T \in L(X)$ be compact, strongly positive with its spectral radius $r(T) > 0$. Then,
\begin{enumerate}
\item[(a)] $r(T) > 0$ is a geometrically simple eigenvalue of $T$ pertaining to an eigenvector $x_0 \in \text{int}(P).$
\item[(b)] For all $ \lambda \neq r(T)$, we have $|\lambda| < r(T)$.
\item[(c)] Any real eigenvalue of $T$ with positive eigenvector must be $r(T)$.
\end{enumerate}
\end{proposition}
Under the conditions $c \leq w(x,y) \leq C,$ $T_w$ is a strongly positive, compact operator.  Moreover, for any normal operator (which includes self-adjoint operators), the only possible accumulation point of the spectrum is $0$ (see for example, Theorem 12.30 of \citet{rudin-functional-analysis}).  Thus, there must exist a gap $\zeta = 1- \max(\lambda_{2}(T_w), |\lambda_{-1}(T_w)|) > 0$.  

Letting $\gamma = 1 - \zeta/2$, we have that for any $\delta >0$, there exists some $N_\delta$ such that for all $n > N_\delta$:
\begin{align*}
& \ 1- P\left( \forall \ t \geq 1, \ \ \|\widetilde{P}_n^t(\ \cdot \ | \ \mathcal{X}_0 = i) - \pi(\cdot) \|_{TV} \leq \sqrt{Mn} \gamma^t   \right) 
\\ \leq & \ P\left( \max_{i \in \{-1, 2\}}\left|\lambda_i(\mathscr{A}) -\lambda_i(T_w)\right| > \zeta/2 \right) + P\left(\max_{1 \leq i \leq n} \frac{1}{n\pi(i)} > M \right)
\\  \leq & \ \delta.
\end{align*}
The result follows.
\qed
\end{appendix}
\newpage

\bibliographystyle{apalike} 
\bibliography{mybib.bib}

\end{document}